\documentclass{amsart}
\RequirePackage[utf8]{inputenc}
\usepackage{amsmath,amsthm,amsfonts,amssymb,amscd,amsbsy,dsfont,multirow, hyperref, color}
\usepackage[all]{xy}

\newcommand{\gpr}{g_\mathrm{prod}}
\newcommand{\gr}{g_\mathrm{round}}
\renewcommand{\v}{\mathrm{v}}
\newcommand{\dd}{\mathrm{d}}

\newcommand{\vol}{\operatorname{vol}}
\newcommand{\Vol}{\operatorname{Vol}}

\newcommand{\Met}{\operatorname{Met}}

\newcommand{\spec}{\operatorname{spec}}
\newcommand{\scal}{\mathrm{scal}}

\renewcommand{\H}{\mathds H}
\newcommand{\R}{\mathds R}
\renewcommand{\S}{\mathds S}
\newcommand{\N}{\mathds N}

\newcommand{\SO}{\mathsf{SO}}
\newcommand{\gen}{\operatorname{gen}}

\hypersetup{
    pdftoolbar=true,
    pdfmenubar=true,
    pdffitwindow=false,
    pdfstartview={FitH},
    pdftitle={Bifurcation of periodic solutions to the singular Yamabe problem on spheres},
    pdfauthor={Renato G. Bettiol, Paolo Piccione and Bianca Santoro},
    pdfsubject={Primary: 58J55, 58E11, 58E15; Secondary: 53A30, 53C21, 35J60},
    pdfkeywords={}
    pdfnewwindow=true,
    colorlinks=true,
    linkcolor=blue,
    citecolor=blue,
    urlcolor=blue,
}

\newtheorem{theorem}{Theorem}[]
\newtheorem{lemma}[theorem]{Lemma}
\newtheorem{proposition}[theorem]{Proposition}
\newtheorem{corollary}[theorem]{Corollary}

\newtheorem*{prob}{Singular Yamabe Problem}

\newtheorem*{mainthm}{Main Theorem}

\newtheorem*{theorem*}{Theorem}

\theoremstyle{definition}

\theoremstyle{remark}
\newtheorem{remark}[theorem]{Remark}

\title[Bifurcation of singular Yamabe problem on spheres]{Bifurcation of periodic solutions to the singular Yamabe problem on spheres}

\author[R. G. Bettiol]{Renato G. Bettiol}
\author[P. Piccione]{Paolo Piccione}
\author[B. Santoro]{Bianca Santoro}

\address{
\begin{tabular}{lll}
University of Pennsylvania & &Universidade de S\~ao Paulo \\
David Rittenhouse Lab. & & Departamento de Matem\'atica \\
Department of Mathematics & & Rua do Mat\~ao, 1010 \\
209 South 33rd St & & S\~ao Paulo, SP, 05508-090, Brazil\\
Philadelphia, PA, 19104-6395, USA & & \emph{E-mail address}: {\tt piccione@ime.usp.br}\\
\emph{E-mail address}: {\tt rbettiol@math.upenn.edu} && \\
\end{tabular}
\bigskip
\hfill\break\hfill\indent
\begin{tabular}{lll}
The City College of New York and Graduate Center, CUNY&&\\
Mathematics Department&&\\
138th St and Convent Avenue, NAC 4/112B&&\\
New York, NY, 10031, USA&&\\
\emph{E-mail address}: {\tt bsantoro@ccny.cuny.edu} & &
\end{tabular}
}

\allowdisplaybreaks
\numberwithin{equation}{section}
\numberwithin{theorem}{section}

\thanks{The first named author is supported by the NSF grant DMS-1209387, USA. The second named author is partially supported by Fapesp and CNPq, Brazil. The third named author is partially supported by the NSF grant DMS-1007155 and PSC-CUNY grants, USA}

\date{\today}
\begin{document}
\begin{abstract}
We obtain uncountably many periodic solutions to the singular Yamabe problem on a round sphere, that blow up along a great circle. These are (complete) constant scalar curvature metrics on the complement of $\S^1$ inside $\S^m$, $m\geq5$, that are conformal to the round (incomplete) metric and \emph{periodic} in the sense of being invariant under a discrete group of conformal transformations. These solutions come from bifurcating branches of constant scalar curvature metrics on compact quotients of $\S^m\setminus \S^1\cong \S^{m-2}\times \H^{2}$.
\end{abstract}

\maketitle

\section{Introduction}
A major achievement in Geometric Analysis was the complete solution of the Yamabe problem, which asserts that any compact Riemannian manifold $(M,g)$ of dimension $\dim M=m\geq3$ admits a metric $g'$ that has constant scalar curvature and is conformal to $g$. Attempts to generalize this statement in many directions have provided a fertile research area in the last decades. One such direction is to drop the compactness assumption on $M$. In this case, it is natural to only consider \emph{complete} metrics, even though every conformal class obviously contains both complete and incomplete metrics. A well-studied version of this question, where the geometry at infinity is relatively tame, is the following:

\begin{prob}
Let $(M,g)$ be a compact Riemannian manifold and $\Lambda\subset M$ be a closed subset. Find a complete metric $g'$ on $M\setminus\Lambda$ that has constant scalar curvature and is conformal to $g$.
\end{prob}

Although considerable progress was made in the general case, for many reasons the above problem is especially interesting in the particular case in which $(M,g)$ is the unit round sphere $(\S^m,\gr)$, see Schoen~\cite[\textsection 5]{schoen89}. This situation was initially considered in the 1970's by Loewner and Nirenberg \cite{ln}, that obtained existence of solutions with $\scal<0$ in some cases where the Hausdorff dimension of $\Lambda$ is $\geq(m-2)/2$. Most of the subsequent contributions to the problem, up to present days, assume that $\Lambda$ is smooth, with remarkable exceptions due to Finn~\cite{finn1,finn2}. Under this assumption, Aviles and McOwen \cite{am2} proved that for a general $(M,g)$, a solution with $\scal<0$ exists if and only if $\dim \Lambda >(m-2)/2$, in which case one also has uniqueness of solutions and regularity results, see Mazzeo~\cite{mazzeo}.

Analogously to the classical Yamabe problem, the case $\scal\geq0$ is considerably more involved. The first major breakthroughs were obtained by Schoen \cite{schoen88} and Schoen and Yau \cite{sy}. The latter established that if $\S^m\setminus\Lambda$ admits a complete metric with scalar curvature bounded below by a positive constant, then the Hausdorff dimension of $\Lambda$ is $\leq(m-2)/2$, and the former constructed several examples of domains $\S^m\setminus\Lambda$ that admit complete conformally flat metrics with constant positive scalar curvature, including the case where $\Lambda$ is any finite set with at least two points. This existence result was greatly generalized by Mazzeo and Pacard~\cite{mp1,mp2}, allowing $\Lambda$ to be a disjoint union of submanifolds with dimensions between $1$ and $(m-2)/2$ when $(M,g)$ is a general compact manifold with constant nonnegative scalar curvature, and between $0$ and $(m-2)/2$ in case $(M,g)=(\S^m,\gr)$.

Some of the first solutions to the Singular Yamabe Problem on $\S^m\setminus \Lambda$ were constructed by lifting solutions to the classical Yamabe problem from compact quotients. In these constructions, $\Lambda$ is the limit set of a Kleinian group, hence either a round subsphere $\S^k\subset \S^m$ or totally unrectifiable. Since the corresponding metrics on $\S^m\setminus \Lambda$ are invariant under a discrete (cocompact) group of conformal transformations, we slightly abuse terminology and call these \emph{periodic solutions}.

The purpose of the present paper is to apply bifurcation techniques to obtain many families of new periodic solutions when $\Lambda=\S^1$. Our central result is:

\begin{mainthm}
There exist uncountably many \emph{branches} of periodic solutions to the singular Yamabe problem on $\S^m\setminus \S^1$, for all $m \geq 5$, having (constant) scalar curvature arbitrarily close to $(m-4)(m-1)$.
\end{mainthm}

The starting point for constructing such solutions is the existence of \emph{trivial} periodic solutions in the case $\Lambda$ is a round subsphere $\S^k\subset \S^m$, cf.\ \cite{mazzeo-smale}. Through stereographic projection using a point in $\S^k$, there is a conformal equivalence $\S^m\setminus \S^k\cong \R^m\setminus \R^k$. Consider the flat metric in the latter, given in cylindrical coordinates by $\dd r^2+r^2\dd \theta^2+\dd y^2$, where $y$ is the coordinate in $\R^k$ and $(r,\theta)$ are polar coordinates in its orthogonal complement $\R^{m-k}$. Dividing by $r^2$, we obtain that this metric is conformal to
\begin{equation*}
\gpr=\dd \theta^2+\frac{\dd r^2+\dd y^2}{r^2}=\frac{\dd r^2+r^2\dd \theta^2+\dd y^2}{r^2},
\end{equation*}
which is clearly the ordinary product metric on $\S^{m-k-1}\times \H^{k+1}$, with the hyperbolic metric written in the upper half-plane model.
Alternatively, recall that the subset of $\S^m$ at maximal distance from any round subsphere $\S^k$ is another round subsphere $\S^{m-k-1}$, and both submanifolds have trivial normal bundle.\footnote{This is the usual decomposition of $\S^m$ as the spherical join $\S^k * \S^{m-k-1}$.} In particular, the exponential map is a diffeomorphism between $\S^m\setminus \S^k$ and the (trivial) normal disk bundle $D(\S^{m-k-1})=\S^{m-k-1}\times D^{k+1}$. The pull-back of $\gr$ is a doubly warped product metric on $[0,\tfrac{\pi}{2})\times\S^{m-k-1}\times\S^k$. Dividing by the warping function of $\S^{m-k-1}$ and reparametrizing the $[0,\tfrac{\pi}{2})$ coordinate, one reobtains the product metric on $\S^{m-k-1}\times\H^{k+1}$, with the hyperbolic metric now written in the rotationally symmetric (Poincar\'e disk) model.

Thus, $\gpr$ is a (smooth) solution to the Singular Yamabe Problem, since it is conformal to the (incomplete) round metric on $\S^m\setminus \S^k$ and has constant scalar curvature equal to:
\begin{equation*}
\scal_{m,k}:=(m-k-1)(m-k-2)-(k+1)k=(m-2k-2)(m-1).
\end{equation*}
In particular, notice that $\scal_{m,k}>0$ precisely when $k<(m-2)/2$. Mazzeo and Smale~\cite{mazzeo-smale} used these trivial solutions to prove existence of infinitely many other solutions with scalar curvature $\scal_{m,k}$, perturbing the subsphere $\S^k\subset \S^m$ with a diffeomorphism of $\S^m$ close to the identity. These were the first nonperiodic solutions to be found on $\S^m\setminus \S^k$.

In order to find new periodic solutions, we concentrate on the exceptional case $k=1$, which allows for paths of nonisometric\footnote{If $k\geq2$, then any compact quotients $\S^{m-k-1}\times\Sigma^{k+1}$ of $\S^m\setminus\S^k$ that have isomorphic fundamental groups must be isometric, by Mostow's Rigidity Theorem.} compact quotients of $\S^m\setminus \S^1\cong \S^{m-2}\times \H^{2}$. These arise from paths of (cocompact) lattices $\Gamma_t$ on $\H^2$, which induce paths of hyperbolic surfaces $\Sigma_t=\H^2/\Gamma_t$, whose product with $\S^{m-2}$ gives paths of closed manifolds $\S^{m-2}\times\Sigma_t$.
We prove that if the path $\Sigma_t$ of hyperbolic surfaces degenerates in an appropriate way, then there exist nonproduct constant scalar curvature metrics on $\S^{m-2}\times \Sigma$ which \emph{bifurcate} from $\S^{m-2}\times\Sigma_t$. These lift to metrics on $\S^m\setminus \S^1$ which are conformal to $\gpr$ (and hence to $\gr$), providing new periodic solutions with \emph{varying period} $\Gamma_t$. We obtain a very large quantity of solutions, with very general periods, since for any lattice $\Gamma_0$ on $\H^2$ we find a path $\Gamma'_t$ of lattices (starting arbitrarily close to $\Gamma_0$) such that bifurcation occurs along $\S^{m-2}\times(\H^2/\Gamma'_t)$, see Theorem~\ref{thm:curvesbifurcating}. Moreover, this can be done with lattices corresponding to hyperbolic surfaces of any desired genus $\geq2$.

Our bifurcating solutions have positive scalar curvature, and since $\scal_{m,1}>0$ if and only if $m\ge5$, the range for $m$ considered in the above theorem is the largest possible. Regarding the case $m=4$, notice that $\scal_{4,1}=0$ and hence any two constant scalar curvature metrics on a compact quotient of $(\S^{2}\times \H^2,\gpr)$ must be homothetic \cite[p.\ 175]{aubin}. Thus, $\gpr$ is the unique periodic solution on $\S^4\setminus \S^1$.

Bifurcation of constant scalar curvature metrics from paths of product metrics on $\S^{m-2}\times \Sigma_t$ is obtained via a classical variational bifurcation criterion in terms of the Morse index of such metrics (as critical points of the Hilbert-Einstein functional). This Morse index can be computed explicitly in terms of eigenvalues of the Laplacian of $\Sigma_t$. Informed by well-known spectral results for hyperbolic surfaces, we choose an appropriate path $\Sigma_t$ (which corresponds geometrically to pinching a noncontractible closed geodesic) to produce the desired effect on the Morse index that yields bifurcation. The only relevant information about the factor $\S^{m-2}$ is that the first eigenvalue of its Laplacian satisfies a certain bound in terms of scalar curvature, see \eqref{eq:basicinequality}. The same bifurcation result holds replacing $\S^{m-2}$ by any other closed manifold $N$ that satisfies \eqref{eq:basicinequality}, see Theorem~\ref{thm:curvesbifurcating}.
In principle, this provides a method to construct periodic solutions to the Singular Yamabe problem on $M\setminus\Lambda$, for any $M$ that admits a codimension $2$ submanifold $N$ satisfying \eqref{eq:basicinequality}, whose normal bundle is trivial and conformally equivalent to $M\setminus\Lambda$. For this reason, we prove all relevant bifurcation results on a general manifold $N\times\Sigma$, and only specialize to $N=\S^{m-2}$ in the last section of the paper. Determining other geometrically interesting occurrences of the above situation is an object of ongoing research by the authors.

The paper is organized as follows. In Section~\ref{sec:spec}, we recall some spectral properties of hyperbolic surfaces. The appropriate variational framework for finding metrics of constant scalar curvature is discussed in Section~\ref{sec:bif}, and the core bifurcation result (Theorem~\ref{thm:curvesbifurcating}) is proved. Finally, Section~\ref{sec:proofs} contains the final arguments necessary to prove the above Main Theorem.

\medskip
\noindent\textbf{Acknowledgements.} It is a pleasure to thank Rafe Mazzeo for introducing us to the Singular Yamabe Problem and suggesting it as a possible application of our bifurcation techniques. We would also like to thank Sugata Mondal for bringing reference \cite{wolpert} to our attention.

\section{Spectrum of hyperbolic surfaces}
\label{sec:spec}

Let $\Sigma$ be a closed oriented smooth surface of genus $\gen(\Sigma)\geq2$. Recall that, by the Uniformization Theorem, every conformal class of metrics on $\Sigma$ contains a unique representative of constant sectional curvature $-1$,  called a \emph{hyperbolic metric}. We denote by $\mathcal H(\Sigma)$ the space of smooth hyperbolic metrics on $\Sigma$, endowed with the Whitney $C^\infty$ topology.\footnote{Observe that, in this particular situation, the choice of Whitney $C^\infty$ topology still allows for the use of results usually available only for Banach manifolds. More precisely, denote by $\mathcal H_r(\Sigma)$ the Banach manifold of $C^r$ hyperbolic metrics on $\Sigma$, with the Whitney $C^r$ topology, and by $\mathrm{Diff}_r^+(\Sigma)$ the group of orientation preserving $C^r$ diffeomorphisms of $\Sigma$. The orbit space $\mathcal T(\Sigma)=\mathcal H_r(\Sigma)/\mathrm{Diff}_r^+(\Sigma)$, called the \emph{Teichm\"uller space} of $\Sigma$, is independent of $r$ and homeomorphic to $\R^{6\gen(\Sigma)-6}$, see \cite{Tro92}. Given any smooth hyperbolic metric $g_0\in\mathcal H_r(\Sigma)$, there exists a $(6\gen(\Sigma)-6)$-dimensional submanifold $\mathcal S$ of $\mathcal H_r(\Sigma)$, consisting of smooth hyperbolic metrics, such that the restriction to $\mathcal S$ of the quotient map $\pi\colon\mathcal H_r(\Sigma)\to\mathcal T(\Sigma)$ is a smooth diffeomorphism onto a neighborhood of $\pi(g_0)\in\mathcal T(\Sigma)$. Thus, results on Banach manifolds can also be applied to $\mathcal H(\Sigma)$, provided they are local and invariant under diffeomorphisms.}
The Riemannian manifold $(\Sigma,h)$, with $h\in\mathcal H(\Sigma)$, will be called a \emph{hyperbolic surface}, and we denote the spectrum of the Laplacian operator on real-valued functions on $(\Sigma,h)$ by
\begin{equation*}
\spec(\Sigma,h)=\big\{0<\lambda_1(\Sigma,h)\leq\lambda_2(\Sigma,h)\leq\dots\leq\lambda_k(\Sigma,h) \nearrow+\infty\big\},
\end{equation*}
where each eigenvalue is repeated according to its multiplicity. The goal of this section is to discuss some spectral properties of hyperbolic surfaces required for our applications; for general references see \cite{buser-book,chavel}.

\subsection{Small eigenvalues}
One of the cornerstones of the proof of our Main Theorem is the behavior of eigenvalues of the Laplacian of hyperbolic surfaces near $\frac14$.
To illustrate the peculiar nature of these eigenvalues, we recall that classical estimates imply that, for any $\gen(\Sigma)\geq2$, there exist hyperbolic metrics on $\Sigma$ whose first $2\gen(\Sigma)-3$ eigenvalues are arbitrarily close to zero. On the other hand, the long-standing conjecture that at most $2\gen(\Sigma)-2$ eigenvalues can be $\leq\frac14$ for any hyperbolic metric on $\Sigma$ was recently proved by Otal and Rosas~\cite{or}. A well-known fact related to the above is that arbitrarily many eigenvalues can lie in the interval $[\tfrac14,\frac14+\varepsilon]$, see  \cite[Thm~8.1.2, p.\ 211]{buser-book} or \cite[Thm 2, p.\ 251]{chavel}. More precisely, the following holds:

\begin{proposition}\label{prop:buser}
Let $\Sigma$ be a closed oriented surface of genus $\gen(\Sigma)\geq2$. For all $\varepsilon>0$ and $k\in\N$, there exists $h\in\mathcal H(\Sigma)$ such that $\lambda_k(\Sigma,h)<\tfrac14+\varepsilon$.
\end{proposition}

The above implies that, up to deforming the hyperbolic metric $h\in\mathcal H(\Sigma)$,
\begin{equation*}
n_a(\Sigma,h):=\max \big\{k\in\N:\lambda_k(\Sigma,h)<a\big\}
\end{equation*}
can be made arbitrarily large for any given $a>\tfrac14$. More precisely, we have:

\begin{corollary}\label{cor:ma}
Let $(\Sigma,h_0)$ be a hyperbolic surface. For any $a>\tfrac14$ and any nonnegative integer $d$, there exists a real-analytic path of hyperbolic metrics $h_t\in \mathcal H(\Sigma)$, $t\in [0,1]$, such that $n_a(\Sigma,h_1)>n_a(\Sigma,h_0)+d$.
\end{corollary}

\begin{proof}
Let $K=\min \{k:\lambda_k(\Sigma,h_0)\geq a\}$. Since $a>\tfrac14$, by Proposition~\ref{prop:buser}, there exists $h_1\in\mathcal H(\Sigma)$ such that $\lambda_{K+d}(\Sigma,h_1)<a$. Thus $n_a(\Sigma,h_0)=K-1$ and $n_a(\Sigma,h_1)\geq K+d$. Since $\mathcal H(\Sigma)$ is a real-analytic path-connected\footnote{Since $\mathcal T(\Sigma)=\mathcal H(\Sigma)/\mathrm{Diff}^+(\Sigma)$ and $\mathrm{Diff}^+(\Sigma)$ are path-connected, so is $\mathcal H(\Sigma)$.} manifold, the conclusion follows.
\end{proof}

\subsection{Avoiding eigenvalues}
In our applications, a crucial step is to avoid a given real number $\lambda>\tfrac14$ in $\spec(\Sigma,h)$, up to perturbing the hyperbolic metric $h\in\mathcal H(\Sigma)$. In this context, it is natural to consider the effect of real-analytic deformations of hyperbolic metrics on the corresponding eigenvalues of the Laplacian.

Given a real-analytic path $h_\ell\in\mathcal H(\Sigma)$, the corresponding Laplacians $\Delta_{h_\ell}$ form a real-analytic path of symmetric unbounded discrete operators. By the Kato Selection Theorem~\cite{Kato95}, their eigenvalues are real-analytic functions of $\ell$, up to relabeling. More precisely, given $\lambda\in\spec(\Sigma,h_{\ell_0})$, there are real-analytic functions $\lambda(\ell)$, called \emph{eigenvalue branches} through $\lambda$, such that $\lambda(\ell_0)=\lambda$ and $\lambda(\ell)\in\spec(\Sigma,h_\ell)$. For $\ell$ near $\ell_0$, these are the only elements of $\spec(\Sigma,h_\ell)$ near $\lambda$. Notice that, since $\lambda(\ell)$ are real-analytic, they can only attain the value $\lambda$ countably many times.

In this framework, we now prove the desired avoidance principle.

\begin{proposition}\label{prop:lambdainspectrum}
Let $\Sigma$ be a closed oriented surface of genus $\gen(\Sigma)\geq2$, and fix $\lambda>\tfrac14$. Then the subset $\mathcal H_\lambda(\Sigma):=\big\{h\in\mathcal H(\Sigma):\lambda\not\in\spec(\Sigma,h)\big\}$ is open and dense.
\end{proposition}

\begin{proof}
The condition $\lambda\not\in\spec(\Sigma,g)$ is open in the space $\Met(\Sigma)$ of \emph{all} smooth Riemannian metrics $g$ on $\Sigma$, hence also in $\mathcal H(\Sigma)$. In order to prove density of this condition, suppose $\lambda\in\spec(\Sigma,h)$, and let $h_\ell\in\mathcal H(\Sigma)$ be a real-analytic path of \emph{pinching} hyperbolic metrics through $h$. In other words, $h=h_{\ell_0}$ for some $\ell_0>0$, and $(\Sigma,h_\ell)$ have shortest closed geodesics of length $\ell$, that are pinched in the limit $\ell\searrow0$. The existence of such paths is proved by Wolpert~\cite[\textsection 2.5]{wolpertII}.
Denote by $\lambda(\ell)$ the corresponding eigenvalue branches through $\lambda$, which are real-analytic functions of $\ell$, as described above. If none of the $\lambda(\ell)$ are constant functions, then we are done, as $\lambda\not\in\spec(\Sigma,h_\ell)$ for any $\ell\neq\ell_0$ near $\ell_0$. Otherwise, there exists a constant eigenvalue branch $\lambda(\ell)\equiv\lambda$; in particular, $\lim_{\ell\searrow0}\lambda(\ell)=\lambda>\tfrac14$. According to a deep result of Wolpert~\cite[Thm 5.14]{wolpert}, the only eigenvalue branches for degenerating hyperbolic surfaces as above whose limit is $>\tfrac14$ are \emph{nonconstant}. This contradiction implies that this latter case cannot happen.
\end{proof}

\begin{remark}
Given $\lambda>\tfrac14$, it is a hard problem to find an \emph{explicit} hyperbolic surface $(\Sigma,h)$ with $\lambda\not\in\spec(\Sigma,h)$. For small $\lambda$, this is related to finding hyperbolic surfaces with large first eigenvalue. Recall that $\lambda_1(\Sigma,h)\leq 2(\gen(\Sigma)+1)/(\gen(\Sigma)-1)$, by an estimate of Yang and Yau~\cite{yy}. In particular, the larger $\gen(\Sigma)$ is, the smaller the upper bound on $\lambda_1(\Sigma,h)$. A well-studied hyperbolic surface of genus $2$ is the \emph{Bolza surface} $(\Sigma_{B},h_B)$, which has the largest systole (shortest noncontractible closed geodesic) and the largest conformal group among such surfaces. Numeric estimates yield $\lambda_1(\Sigma_B,h_B)\cong3.838$, see \cite[\textsection 5.3]{StrUsk2013}; and hence provide an explicit example with $\lambda\not\in\spec(\Sigma_B,h_B)$ for any $\lambda<3.8$.
\end{remark}

\section{Bifurcations from constant scalar curvature product metrics}
\label{sec:bif}

Let $(N,g_N)$ be a compact Riemannian manifold with $\dim N=n$ and constant scalar curvature $\scal_{N}\in\R$, and let $(\Sigma,h)$ be a hyperbolic surface. Denote by $g=g_N\oplus h$ the product metric on $N\times\Sigma$. The product manifold $(N\times\Sigma,g)$ has
\begin{equation}\label{eq:scalvol-product}
\begin{aligned}
\scal_{g}&=\scal_{N}-2,  \\
\Vol(N\times\Sigma,g)&=4\pi(\gen(\Sigma)-1)\Vol(N,g_N).
\end{aligned}
\end{equation}
In this section, we discuss the variational approach to the problem of finding constant scalar curvature metrics on $N\times\Sigma$, and establish our core results on bifurcation of solutions issuing from families of product metrics.

\subsection{Variational setup}
Consider the Sobolev space $H^1(N\times\Sigma)$ and the Lebesgue space $L^p(N\times\Sigma,\mathrm{vol}_g)$, where $\mathrm{vol}_g$ is the volume density of the product metric $g$. By the Gagliardo-Nirenberg-Sobolev inequality, there is a continuous inclusion $H^1(N\times\Sigma)\hookrightarrow L^{\frac{2(n+2)}{n}}(N\times\Sigma,
\mathrm{vol}_g)$, and the subset
\begin{equation*}
[g]_{\v}:=\left\{\begin{array}{r}\phi\in H^1(N\times\Sigma) :
\int_{N\times\Sigma}\phi^{\frac{2(n+2)}n}\,\mathrm{vol}_g=\Vol(N\times\Sigma,g)\\[0.2cm]
\text{and } \phi>0\ \text{a.e.}
\end{array}
\right\}
\end{equation*}
is a smooth Hilbert submanifold of $H^1(N\times\Sigma)$.
The map $[g]_\v\ni\phi\mapsto g_\phi=\phi^\frac4n\, g$ gives an identification between $[g]_\v$ and the set
of Sobolev $H^1$ metrics in the conformal class of $g$ that have the same volume as $g$. The constant map $1\in [g]_{\v}$ clearly corresponds to the original metric $g$, and the tangent space to $[g]_{\v}$ at this point is
\begin{equation}\label{eq:tangentspace}
T_{1}[g]_{\v}=\left\{\psi\in H^1(N\times\Sigma,g):\int_{N\times\Sigma}\psi\,\vol_g=0\right\}.
\end{equation}

It is well-known that $\phi\in [g]_\v$ is a critical point of the \emph{Hilbert-Einstein functional}
$\mathcal A\colon [g]_{\v}\to\R$, defined by
\begin{equation}\label{eq:defHE}
\mathcal A(\phi):=\int_{N\times\Sigma}\Big(4\tfrac{n+1}{n}\,|\nabla\phi|_g^2+(\scal_{N}-2) \,\phi^2\Big)\,\vol_g,
\end{equation}
if and only if $\phi\in C^\infty(N\times\Sigma)$ and $g_\phi=\phi^\frac{4}{n}\, g$ has constant scalar curvature, see e.g.\ \cite{lp,schoen89,trudinger68}. In particular, since the product metric $g$ has constant scalar curvature, the constant function $1\in [g]_\v$ is a critical point of $\mathcal A$. The second variation of $\mathcal A$ at this critical point is the bilinear symmetric form on $T_1[g]_\v$ given by
\begin{equation*}
\dd^2\!\mathcal A(1)(\psi_1,\psi_2)=\frac{n(n+1)}{2}\int_{N\times\Sigma} \Big(g(\nabla\psi_1,\nabla\psi_2)-\frac{\scal_{N}-2}{n+1}\psi_1\psi_2\Big)\!\vol_g.
\end{equation*}
Using the compactness of $H^1(N\times\Sigma)\hookrightarrow L^2(N\times\Sigma)$, we have the existence of a self-adjoint operator $F_g\colon T_1[g]_\v\to T_1[g]_\v$, given by a compact perturbation of the identity, such that
\begin{equation}\label{eq:d2a-strong}
\dd^2\!\mathcal A(1)(\psi_1,\psi_2)=\tfrac{n(n+1)}{2}\big\langle F_g\psi_1,\psi_2\big\rangle_{H^1}, \quad\mbox{ for all } \psi_1,\psi_2\in T_1[g]_\v.
\end{equation}
In particular, $F_g$ is an essentially positive Fredholm operator of index zero. The dimension of $\ker F_g$ and the number (counted with multiplicity) of negative eigenvalues of $F_g$ are, respectively, the \emph{nullity} and the \emph{Morse index} of $1$ as a critical point of $\mathcal A$ in $[g]_\v$. As customary, the second variation of $\mathcal A$ is better understood using an $L^2$-pairing, rather than an $H^1$-pairing, in the space of functions on $N\times\Sigma$ with zero average. Replacing \eqref{eq:tangentspace} with
\begin{equation*}
L^2_0(N\times\Sigma,g):=\left\{\psi\in L^2(N\times\Sigma,\mathrm{vol}_g):\int_{N\times\Sigma}\psi\,\vol_g=0\right\},
\end{equation*}
we can describe $\dd^2\!\mathcal A(1)$ in terms of an unbounded symmetric Fredholm operator $J_g\colon L^2_0(N\times\Sigma,g)\to L^2_0(N\times\Sigma,g)$, by means of
\begin{equation*}\label{eq:d2a-weak}
\dd^2\!\mathcal A(1)(\psi_1,\psi_2)=\tfrac{n(n+1)}{2}\big\langle J_g \psi_1,\psi_2\big\rangle_{L^2}, \quad\mbox{ for all } \psi_1,\psi_2\in  L^2_0(N\times\Sigma,g).
\end{equation*}
The operator $J_g$, called the \emph{Jacobi operator}, is a self-adjoint elliptic operator that can be explicitly computed as
\begin{equation*}
J_g=\Delta_g-\frac{\scal_{N}-2}{n+1}.
\end{equation*}
The kernel and the number of negative eigenvalues of $F_g$ and $J_g$ coincide, so the nullity and Morse index of critical points of the Hilbert-Einstein function can be computed using the spectrum of $J_g$. The latter is given by the (positive) spectrum of the Laplacian $\Delta_g$, shifted to the left by $\frac{\scal_{N}-2}{n+1}$, i.e., the eigenvalues of $J_g$ are:
\begin{equation*}
\lambda_1(N\times\Sigma,g)-\tfrac{\scal_{N}-2}{n+1}\leq
\dots\leq\lambda_k(N\times\Sigma,g)-\tfrac{\scal_{N}-2}{n+1}\leq\dots\nearrow+\infty,
\end{equation*}
where the above are repeated according to multiplicity. This proves the following:

\begin{lemma}\label{lemma:morseidxnull}
The \emph{Morse index} and \emph{nullity} of $1\in[g]_\v$ as a critical point of $\mathcal A\colon [g]_\v\to\R$ are, respectively,
\begin{equation*}
i(g)=\max\left\{k:\lambda_k(N\times\Sigma,g)<\tfrac{\scal_{N}-2}{n+1}\right\} \quad \mbox{and}\quad\nu(g)=\dim\ker J_g.
\end{equation*}
\end{lemma}

\begin{remark}
Note that $-\tfrac{\scal_{N}-2}{n+1}$ is \emph{not} in the spectrum of $J_g$, since the only constant function on $L^2_0(N\times\Sigma,g)$ is identically zero. We also recall the well-known fact that eigenfunctions of $J_g$ are smooth and form an orthonormal basis of $L^2_0(N\times\Sigma,g)$.
\end{remark}

\subsection{Bifurcation}\label{subsec:bif}
Instead of having a fixed hyperbolic metric on $\Sigma$, let $h_t\in\mathcal H(\Sigma)$, $t\in [a,b]$, be a \emph{path} of hyperbolic metrics. Then, we have a corresponding path
\begin{equation}\label{eq:gt}
g_t=g_N\oplus h_t \in\Met(N\times\Sigma), \quad t\in [a,b],
\end{equation}
of product metrics on $N\times\Sigma$ satisfying \eqref{eq:scalvol-product}. We say that $t_*\in [a,b]$ is a \emph{bifurcation instant} for $g_t$ if there exist sequences $\{t_q\}_{q\in\mathds N}$ in $[a,b]$ converging to $t_*$ and
$\{g_q\}_{q\in\mathds N}$ in $\Met(N\times\Sigma)$ converging to $g_{t_*}$, called a \emph{bifurcating branch}, such that:
\begin{itemize}
\item[(i)] each $g_q$ has constant scalar curvature;
\item[(ii)] $g_q$ is conformal to $g_{t_q}$, but $g_q\neq g_{t_q}$;
\item[(iii)] $\Vol(N\times\Sigma,g_q)=4\pi(\gen(\Sigma)-1)\Vol(N,g_N)$.
\end{itemize}
In other words, $t_*$ is a bifurcation instant if local uniqueness of $g_t$ as a solution to \eqref{eq:scalvol-product} fails around $g_{t_*}$. That is, for any open neighborhood of $g_{t_*}\in\Met(N\times\Sigma)$, there are \emph{other} constant scalar curvature metrics (with normalized volume) in the conformal class of some $g_t$, with $t$ near~$t_*$.

We now establish the key result used in our applications, namely, the existence of paths of product metrics on $N\times\Sigma$ along which the Morse index has arbitrarily large variation.

\begin{proposition}\label{prop:morseidx}
Assume that the following inequalities hold:
\begin{equation}\label{eq:basicinequality}
\frac14<\frac{\scal_{N}-2}{n+1}<\lambda_1(N).
\end{equation}
Then, for any fixed $(\Sigma,h_0)$ and any nonnegative integer $d$, there exists a real-analytic path $h_t\in\mathcal H(\Sigma)$ such that, setting $g_t=g_N\oplus h_t$, we have $i(g_1)>i(g_0)+d$.
\end{proposition}

\begin{proof}
Let $a\in\left[\frac{\scal_{N}-2}{n+1},\lambda_1(N)\right]$. From Corollary~\ref{cor:ma}, there exists a real-analytic path $g_t=g_N\oplus h_t$, such that $n_a(\Sigma,h_1)>n_a(\Sigma,h_0)+d$. Note that $i(g_1)\geq n_a(\Sigma,h_1)$. Since $\lambda_1(N)\geq a$, the only eigenvalues of the Laplacian of $(N\times\Sigma,g_0)$ that are strictly smaller than $a$ are of the form $\lambda_0(N)+\lambda_k(\Sigma,h_0)\in\spec(\Sigma,h_0)$, with $\lambda_k(\Sigma,h_0)<a$. In particular, $i(g_0)=n_a(\Sigma,h_0)$. Altogether, we have $i(g_1)\geq n_a(\Sigma,h_1)>n_a(\Sigma,h_0)+d=i(g_0)+d$.
\end{proof}

\begin{theorem}\label{thm:curvesbifurcating}
Let $\Sigma$ be a hyperbolic surface and $N$ be a closed Riemannian manifold such that  \eqref{eq:basicinequality} holds. For any $h_0\in\mathcal H(\Sigma)$, there exist $h'_0,h'_1\in\mathcal H(\Sigma)$ with $h'_0$ arbitrarily close to $h_0$, such that the following holds: for any continuous path $h'_t\in\mathcal H(\Sigma)$ joining $h'_0$ to $h'_1$, there exists at least one bifurcation instant $t_*\in[0,1]$ for the path of constant scalar curvature metrics $g_t=g_N\oplus h'_t$ on $N\times\Sigma$.
\end{theorem}

\begin{proof}
In order to prove the above result, we apply a standard variational bifurcation criterion, adapted to a variable domain framework as in \cite[Appendix~A]{lpz}. This criterion states that, under a local Palais-Smale condition, if $g_0$ and $g_1$ are nondegenerate critical points of different Morse index, then any path of critical points joining $g_0$ to $g_1$ has at least one bifurcation instant, see \cite{bp1,lpz,SmoWas} for details. In what follows, we describe how to verify each of these conditions in the above context of constant scalar curvature product metrics on $N\times \Sigma$.

First, the local Palais-Smale condition is an easy consequence of Fredholmness. Any path of hyperbolic metrics on $\Sigma$ induces a path \eqref{eq:gt} of product metrics on $N\times\Sigma$ and a path $\mathcal A_t\colon [g_t]_\v\to\R$ of smooth functionals given by \eqref{eq:defHE}. We can assume that the constant function $1_t\in [g_t]_\v$ is an isolated critical point for $\mathcal A_t$ for all $t$, otherwise bifurcation trivially holds. Since $\mathrm d^2\!\mathcal A_t(1_t)$ is represented by a Fredholm operator, the first derivative $\dd\mathcal A$ is a \emph{nonlinear Fredholm map} near $1_t$. This implies that each $\mathcal A_t$ satisfies a local Palais-Smale condition\footnote{This follows from a classical argument (see \cite{MarProd}) that uses the local representation for $C^1$ maps having Fredholm derivative as given in \cite[Thm~1.7, p.\ 4]{AbrRob}.} around $1_t$.
More precisely, for any fixed $t_*$, there are $\delta>0$ and a neighborhood $\mathcal U$ of $1_{t_*}$ in $H^1(N\times\Sigma)$ such that $\mathcal A_t$ satisfies the Palais-Smale condition on $\mathcal U\cap[g_t]_\v$ for all $t\in\left[t_*-\delta,t_*+\delta\right]$.

Second, let us describe how to verify nondegeneracy of endpoints, up to small perturbations. For any chosen $h_0\in\mathcal H(\Sigma)$ and nonnegative integer $d$, let $h_t$ be the path given by Proposition~\ref{prop:morseidx}. In particular, this defines $h_1\in\mathcal H(\Sigma)$. Then, Proposition~\ref{prop:lambdainspectrum}, combined with \eqref{eq:basicinequality}, ensures that there exist $h'_0$ and $h'_1$, arbitrarily close to $h_0$ and $h_1$, such that $\frac{\scal_{N}-2}{n+1}\not\in\spec(\Sigma,h'_i)$, $i=0,1$. Consider the metrics
\begin{equation}\label{eq:endpointsgi}
g_i:=g_N\oplus h'_i, \quad i=0,1.
\end{equation}
From \eqref{eq:basicinequality}, the only eigenvalues of the Laplacian of $g_i$ that are $\leq\frac{\scal_{N}-2}{n+1}$ are of the form $\lambda_0(N)+\lambda_k(\Sigma, h'_i)\in\spec(\Sigma,h'_i)$, hence $\frac{\scal_{N}-2}{n+1}\not\in\spec(N\times\Sigma,g_i)$, $i=0,1$. By Lemma~\ref{lemma:morseidxnull}, we have that $\nu(g_i)=0$, i.e., $\ker\dd^2\!\mathcal A_t(g_i)$ is trivial, so \eqref{eq:endpointsgi} are nondegenerate.

Finally, we verify that the Morse index of any path $g_t=g_N\oplus h'_t$ joining $g_0$ to $g_1$ is nonconstant. Again, by \eqref{eq:basicinequality} and Lemma~\ref{lemma:morseidxnull}, the Morse index $i(g_t)$ of $1_t$ as a critical point of $\mathcal A_t$ is given by the number (with multiplicity) of eigenvalues of the Laplacian of $h'_t$ that are strictly smaller than $\frac{\scal_{N}-2}{n+1}$. By Proposition~\ref{prop:morseidx}, and continuity of the spectrum, we have that $i(g_1)>i(g_0)+d$, provided that $h'_0$ and $h'_1$ are chosen sufficiently close to $h_0$ and $h_1$.
\end{proof}

\begin{remark}\label{rem:convergence}
It is not difficult to show that the sequence of metrics that bifurcate from $g_{t_*}$ in Theorem~\ref{thm:curvesbifurcating} converges to $g_{t_*}$ in the $C^r$-topology, for any $r\ge2$.
More generally, suppose $g_q$ is a sequence of smooth metrics on a compact manifold $M$ and $u_q\colon M\to\R_+$ is a sequence of smooth positive functions, such that:
\begin{itemize}
\item $g_q$ has unit volume and constant scalar curvature for all $q$;
\item $g_q\to g_\infty$ in the $C^s$-topology, with $s\ge2$;
\item $u_q\, g_q$ is a unit volume constant scalar curvature metric for all $q$;
\item $u_q\to 1$ in the Sobolev space $H^1(M)$.
\end{itemize}
Then, using $L^p$-estimates for solutions of second order elliptic equations, see \cite[Thm~9.14, p.~240]{GilTru01}, it follows that $\scal(u_q\, g_q)\to\scal(g_\infty)$ and $u_q\to 1$ in the Sobolev space $W^{s+1,p}(M)$, where $p=\frac{2m}{m+2}$ and $m=\dim M$. In particular, if $s>r+\frac{m}{2}$, then $u_q\to 1$ in $C^r(M)$. This applies to the sequence of bifurcating metrics in Theorem~\ref{thm:curvesbifurcating}, since we are assuming that the path $h'_t\in\mathcal H(\Sigma)$ is continuous with respect to the $C^r$-topology for all $r$.
\end{remark}

\section{Proof of Main Theorem}
\label{sec:proofs}

The proof of the Main Theorem in the Introduction is a direct application of Theorem~\ref{thm:curvesbifurcating} to the case $(N,g_N)=(\S^{m-2},\gr)$, $m\geq 5$. In order to verify that \eqref{eq:basicinequality} holds, notice that $n=\dim N=m-2$, and $\scal(\S^{m-2})=(m-2)(m-3)$, hence $\frac{\scal_N-2}{n+1}=m-4$. Furthermore, $\lambda_1(N)=\lambda_1(\S^{m-2})=m-2$. Thus, \eqref{eq:basicinequality} is satisfied for all $m\ge5$.

Applying Theorem~\ref{thm:curvesbifurcating} with arbitrary choices of $h_0\in\mathcal H(\Sigma)$ and continuous paths $h'_t\in \mathcal H(\Sigma)$ joining $h'_0$ to $h'_1$, one obtains the existence of uncountably many bifurcating branches $\{g_q\}_{q\in\mathds N}$ of constant scalar curvature metrics on $\S^{m-2}\times\Sigma$ that have fixed volume and are conformal to a product metric $g_{\text{round}}\oplus h'_{t_q}$. These branches consist of metrics with \emph{positive} constant scalar curvature, since the solution to the Yamabe problem (with volume normalization) is unique in conformal classes with nonpositive conformal Yamabe energy \cite[p.\ 175]{aubin}. From Remark~\ref{rem:convergence}, the values of these scalar curvatures converge to $\scal_{m,1}=(m-4)(m-1)$.

Furthermore, each metric $g_q$ in a bifurcating branch $\{g_q\}_{q\in\mathds N}$ is conformal, but not equal, to $g_{\text{round}}\oplus h'_{t_q}$. This follows from the fact that any two product metrics on a product manifold are conformal if and only if they are homothetic.\footnote{Notice that a metric of the form $g=\phi(x_1,x_2)(g_1\oplus g_2)$ is a product metric $g=g'_1\oplus g'_2$ on the manifold $M_1\times M_2$ if and only if the function $\phi\colon M_1\times M_2\to\R_+$ is constant.}
In particular, two distinct product metrics with the same volume cannot be conformal.
Therefore, $g_q$ cannot be product metrics on $\S^{m-2}\times\Sigma$. Thus, the pull-backs of $g_q$ to $\S^{m-2}\times \H^2$ are complete constant scalar curvature metrics which are conformal, but not equal, to the product metric $\gpr$.
This proves the Main Theorem in the Introduction.\qed

\begin{remark}\label{rem:multiplicity}
The usual notion of \emph{multiple} solutions to the (singular) Yamabe problem
in a conformal class $[g]=\{\phi\,g:\phi\colon M\to\R_+\}$ is that there exist distinct functions $\phi\colon M\to\R_+$ such that $\phi\,g$ has constant scalar curvature. Under this notion of multiplicity, the proof of our Main Theorem guarantees that each bifurcating branch above contains (countably many) pairwise distinct solutions to the Yamabe problem on $\S^{m-2}\times\Sigma$. Since there are continuous families of paths of metrics on $\S^{m-2}\times\Sigma$ that have bifurcating branches, we obtain uncountably many distinct solutions to the Singular Yamabe problem on $\S^m\setminus\S^1$.

In principle, some of these solutions are conformal factors that may give rise to \emph{isometric} metrics. Recall that two distinct metrics in the same conformal class may be isometric, via pull-back by a conformal diffeomorphism. For instance, unlike any other Einstein manifold, the round sphere $(\S^m,g_{\text{round}})$ has uncountably many metrics of constant scalar curvature equal to $\scal(g_{\text{round}})=m(m-1)$, that form a noncompact $(m+1)$-dimensional manifold.
However, all such metrics are isometric.\footnote{A metric in $[g_{\text{round}}]$ has constant scalar curvature if and only if it is the pull-back of $g_{\text{round}}$ by a conformal diffeomorphism of $(S^m,g_{\text{round}})$~\cite[\textsection 2]{schoen89}.
In particular, the moduli space of solutions is diffeomorphic to $\mathrm{Conf}(\S^m,g_{\text{round}})/\mathrm{Iso}(\S^m,g_{\text{round}})\cong\SO(m+1,1)_0/\SO(m+1)$.}
Thus, it is natural to ask whether our result implies the existence of infinitely many pairwise \emph{nonisometric} solutions to the Singular Yamabe problem on $\S^m\setminus \S^1$. Since we are dealing with periodic solutions, this is equivalent to asking whether bifurcating branches of constant scalar curvature metrics on $\S^{m-2}\times\Sigma$ conformal to $g_{\text{round}}\oplus h$ contain nonisometric metrics. Let $\{g^{(1)}_q\}_{q\in\mathds N}$ and $\{g^{(2)}_q\}_{q\in\mathds N}$ be bifurcating branches converging respectively to $g_{\text{round}}\oplus h^{(1)}$ and $g_{\text{round}}\oplus h^{(2)}$, with $h_1,h_2\in\mathcal H(\Sigma)$ distinct (i.e., nonisometric) hyperbolic metrics on $\Sigma$. Since the pull-back action of the diffeomorphism group of $\S^{m-2}\times\Sigma$ on the space of Riemannian metrics is proper (see, e.g.,~\cite[Thm 2.3.1]{Tro92}) and $g_{\text{round}}\oplus h^{(1)}$ and $g_{\text{round}}\oplus h^{(2)}$ are nonisometric, it follows that $g_n^{(1)}$ and $g^{(2)}_n$ are nonisometric for $n$ sufficiently large. Therefore, there are uncountably many pairwise nonisometric periodic solutions to the Singular Yamabe problem on $\S^m\setminus\S^1$. Finally, notice that these solutions can be further chosen to be periodic with respect to infinitely many different cocompact lattices $\Gamma$, for instance using the infinitely many possible choices of $\gen(\Sigma)=\gen(\H^2/\Gamma)\geq2$.
\end{remark}

\begin{remark}\label{rem:oldthmb}
A somewhat weaker nonuniqueness result (independent of bifurcation theoretic methods) follows as a by-product of the above proof and the solution to the classical Yamabe problem. Namely, for any $m\geq5$, we obtain uncountably many hyperbolic metrics $h\in\mathcal H(\Sigma)$ such that the Morse index $i(\gr\oplus h)$ is arbitrarily large, see Proposition~\ref{prop:morseidx}. Hence, the corresponding \emph{Yamabe metric}\footnote{i.e., the metric with the same volume as $\gr\oplus h$ at which the Hilbert-Einstein functional attains its \emph{global minimum} in the conformal class of $\gr\oplus h$.} $g_\text Y(h)$  on $\S^{m-2}\times\Sigma$ must be a different constant scalar curvature metric conformal to $\gr\oplus h$. Arguing as above, the pull-back of $g_\text Y(h)$ to $\S^{m-2}\times\H^2$ is a periodic solution with $\scal>0$ that is not locally isometric to $\gpr$. Note that $g_\text Y(h)\neq g_\text Y(h')$ if $h\neq h'$. Thus, there are uncountably many distinct periodic solutions to the singular Yamabe problem on $\S^m\setminus \S^1$.
\end{remark}

\end{document}